\documentclass[12pt, a4paper, oneside]{amsart}
\usepackage{amssymb, amsmath, amsthm, verbatim, amsbsy,cite}
\usepackage[english]{babel}
\usepackage{amsaddr}
\newtheorem{theorem}{Theorem}[section]
\newtheorem{cor}[theorem]{Corollary}
\newtheorem{lemma}[theorem]{Lemma}
\newtheorem{prop}[theorem]{Proposition}

\newcommand{\cs}{\operatorname{cs}}

\newcommand{\mZ}{\mathbb{Z}}

\def\ov{\overline}
\def\wh{\widehat}

\begin{document}

\title[Characterization of groups by the order and conjugacy class sizes]{Characterization of the alternating and symmetric groups by the order and conjugacy class sizes}

\author{Ilya Gorshkov}
\address{Sobolev Institute of Mathematics, Novosibirsk, Russia}
\email{ilygor8@gmail.com}

\author{Andrey V. Vasil'ev}
\address{Sobolev Institute of Mathematics, Novosibirsk, Russia}
\email{vasand@math.nsc.ru}

\begin{abstract} We prove that an arbitrary finite group $G$ having the same order and same set of conjugacy class sizes as an alternating or symmetric group $S$ must be isomorphic to~$S$. From this and previously known results it follows that the same holds true for every simple group~$S$.\smallskip

\smallskip

\end{abstract}


\maketitle

\section{Introduction}

Given a finite group $G$, we set $\cs(G)=\{|x^G| : x\in G\}$ for the set of conjugacy class sizes of~$G$, and $\cs^*(G)$ for the same collection of parameters but counting their multiplicities. For example, $\cs(A_5)=\{1,12,15,20\}$ but $\cs^*(A_5)=\{\{1,12,12,15,20\}\}$, as there are two conjugacy classes of elements of order $5$ in the alternating group~$A_5$.

It is quite natural to wonder how the structures of these sets reflect on the structure of a~group. The recent example can be found in \cite{25L2}. There for every special irreducible representation $E$ of a Weyl group, Lusztig defined a finite group $\Delta_E$ satisfying some natural conditions, see also \cite{25L1} for details. To establish the correctness of the definition he had to check that the symmetric groups $S_n$ for $n\leq5$ are uniquely determined by $\cs^*(S_n)$. In fact, as we show below, the same assertion holds true for all positive integers~$n$. Moreover, it remains true under much weaker hypothesis.

\begin{theorem}\label{t:main}
Let $S$ be a finite alternating or symmetric group. If $G$ is a group with $|G|=|S|$ and $\cs(G)=\cs(S)$, then $G\simeq S$.
\end{theorem}

Since the order of a finite group is the sum of the sizes of its conjugacy classes, the following holds.

\begin{cor}\label{c:main}
Let $S$ be a finite alternating or symmetric group. If $G$ is a group with $\cs^*(G)=\cs^*(S)$, then $G\simeq S$.
\end{cor}

As the reader can see below, our proof of Theorem~\ref{t:main} is relatively short. However it is based on a quite involved technique developed to verify a conjecture attributed to J.~G.~Thompson, see \cite[Problem~12.38]{KT}. The conjecture claims that every finite nonabelian simple group $S$ is uniquely determined by the set $\cs(S)$ in the class of all finite groups with trivial center.

The Thompson conjecture was proved for all simple groups except the alternating groups \cite{19GorS}, while in the case of the alternating groups it was done only modulo the binary Goldbach conjecture \cite{19GorA}. Nevertheless, the main result of \cite{19GorA} (see also Proposition~\ref{p:GorMain} below) guarantees that a centerless finite group $G$ with $\cs(G)=\cs(S)$, where $S$ is an alternating or symmetric group of degree $n\geq23$, has a composition factor isomorphic to~$A_m$, where $m$ belongs to the interval $[p,n]$ and $p$ is the greatest prime not exceeding~$n$.

In order to apply the latter assertion in our case, we prove (see Lemma~\ref{l:ZofG}) that a group $G$ satisfying the equalities $\cs(G)=\cs(S)$ and $|G|=|S|$ for an alternating or symmetric group $S$ of degree $n\geq4$ must have the trivial center. This does not prove Theorem~\ref{t:main} yet, but makes the proof sufficiently straight.

The proof of Lemma~\ref{l:ZofG} is based on the fact that for every prime divisor $r$ of the order of an alternating or symmetric group $S$ there is an element $x\in S$ such that the $r$-parts of $|x^S|$ and $|S|$ are the same (Lemma~\ref{l:G26cor}). The last fact holds true for every other nonabelian simple group $S$ as it follows from the description of adjacency in the prime graph $GK(S)$ (also called the Gruenberg-Kegel graph) in \cite{05VasVd}. Therefore, the equalities $\cs(G)=\cs(S)$ and $|G|=|S|$ yield $Z(G)=1$ for every nonabelian simple group $S$. Adding the validity of the Thompson conjecture for the nonabelian simple groups distinct from the alternating groups and Theorem~\ref{t:main} to these arguments, we obtain the following.

\begin{cor}\label{c:simple1}
Let $S$ be a finite simple group. If $G$ is a group with $\cs(G)=\cs(S)$ and $|G|=|S|$, then $G\simeq S$.
\end{cor}

\begin{cor}\label{c:simple2}
Let $S$ be a finite simple group. If $G$ is a group with $\cs^*(G)=\cs^*(S)$, then $G\simeq S$.
\end{cor}

\section{Preliminaries}

The {\em cycle structure} of a permutation $x\in S_n$ is the list of the lengths of its cycles arranged in non-increasing order, i.\,e., the record
\begin{equation}\label{eq:perm}
x=(1^{i_1},2^{i_2},\ldots,n^{i_n})
\end{equation}
means that $x$ has $i_k$ cycles of length $k$ for $k=1,\ldots,n$, and $\sum_{k=1}^n k\,i_k=n$. Sometimes for brevity we do not mention missing cycles in the list, e.\,g., instead of $x=(1^0,\ldots,(n-1)^0,n^1)$ we write simply $x=(n^1)$.

\begin{lemma}\label{l:cenSn}
Let $x=(1^{i_1},2^{i_2},\ldots,n^{i_n})\in S_n$. Then
$$
C=C_{S_n}(x)\simeq \bigotimes_{k=1}^n(\mZ_{k}\wr S_{i_k}),\,\, |C|=\prod_{k=1}^n k^{i_k}\cdot k!\mbox{ and } |x^{S_n}|=n!/|C|,
$$
where $\mZ_{k}\wr S_{i_k}$ is supposed to be the trivial group for each $i_k$ equal to~$0$.
\end{lemma}

\begin{proof} See, e.g., \cite[Theorem~1.15]{CB}.
\end{proof}

\begin{lemma}\label{l:cenAn}
Let $x=(1^{i_1},2^{i_2},\ldots,n^{i_n})\in A_n$. Then $C_{A_n}(x)=C_{S_n}(x)$ if and only if $i_k=0$ for all even $k$ and $i_k\leq1$ for all odd $k$, where $k=1,\ldots,n$.
\end{lemma}

\begin{proof} The proof is clear.
\end{proof}

\begin{lemma}\label{l:G26cor} Suppose that $V_n\in\{A_n, S_n\}$ and $n\geq4$. Then for every prime divisor $r$ of $|V_n|$, there is $x\in V_n$ such that the $r$-parts of $|x^{V_n}|$ and $|V_n|$ are the same.
\end{lemma}

\begin{proof} Suppose that $n=\sum_{j=0}^{t} 2^{j}\,b_j$ is the binary representation of $n$. Put $x=(1^{a_1},2^{a_2},\ldots,n^{a_n})$, where $a_k=b_j$, if $k=2^j, j=0,\ldots,t$, and $a_k=0$ otherwise. Lemma~\ref{l:cenSn} implies that $|C_{S_n}(x)|$ is a power of $2$, so the $r$-parts of $|x^{S_n}|$ and $|S_n|$ are the same for every odd prime~$r$.

If $x\in A_n$, then $|C_{A_n}(x)|$ divides $|C_{S_n}(x)|$, so it is also a power of $2$. If $x\in S_n\setminus A_n$ but $n\neq\wh{n}=\sum_{j=0}^{t}2^{j}$, then there is $j\in\{1,\ldots,t\}$ such that $b_j=1$ and $b_{j-1}=0$ in the binary representation of $n$. Replacing the cycle $c$ of length $2^j$ in $x$ with the product $\wh{c}$ of two cycles of lengths $2^{j-1}$ such that $c$ and $\wh{c}$ have the same support, we come to the new element $\wh{x}$. It is clear that $\wh{x}\in A_n$ and $|C_{S_n}(\wh{x})|$ is a $2$-power, whence the $r$-parts of $|\wh{x}^{A_n}|$ and $|A_n|$ are the same for every odd prime~$r$. If $n=\wh{n}$, then replacing cycles of lengths $2$ and $1$ in $x$ with the cycle of length $3$, we obtain an even permutation $\wh{x}$ such that the $r$-parts of $|\wh{x}^A_n|$ and $|A_n|$ coincide for each odd prime $r\neq3$. In the same situation, to get an element which centralizer is not a multiple of $3$, it suffices to replace the cycles of lengths $4$ and $1$ in $x$ with the cycle of length~$5$.


In order to find an element $x$ of $V_n$ such that the $2$-parts of $|x^{V_n}|$ and $|V_n|$ are the same, let $t\in\{n-1,n\}$ be even. Then $t\geq4$. Set
$$
x=\begin{cases}
(3^{1},(t-3)^1),\mbox{ if }\gcd(3,t)=1;\\
(5^{1},(t-5)^1),\mbox{ if }3\,|\,t\mbox{ and }\gcd(5,t)=1;\\
(3^1,5^1,(t-9)),\mbox{ if }15\,|\,t=n;\\
(3^1,5^1,(t-7)),\mbox{ if }15\,|\,t=n-1.
\end{cases}
$$
Observe that $x$ is an even permutation. Applying Lemma~\ref{l:cenSn}, it is easy to check that the centralizer of $x$ in $S_n$ has an odd order, so the $2$-part of $|x^{V_n}|$ and $|V_n|$ are the same, and we are done.
\end{proof}

\begin{lemma}\label{l:ZofG}
Let $V_n\in\{A_n, S_n\}$ and $n\geq4$. If $G$ is a group with $|G|=|V_n|$ and $\cs(G)=\cs(V_n)$, then $Z(G)=1$.
\end{lemma}

\begin{proof}
Suppose that $Z(G)\neq1$. Then there is a prime divisor $r$ of $|Z(G)|$. It follows that the $r$-part of every element in $\cs(G)$ is strictly less than the $r$-part of $|G|=|V_n|$. This contradicts Lemma~\ref{l:G26cor}.
\end{proof}

\begin{lemma}\label{l:cenp}
Let $V_n\in\{A_n, S_n\}$ and $n\geq4$, $x\in V_n$, and $C=C_{V_n}(x)$. Suppose that $p$ is a prime such that $n/2<p\leq n$ and $p$ divides $|C|$. The either
\begin{enumerate}
\item $|C|=p|C_{V_{n-p}}(y)|$ for some $y\in V_{n-p}=S_{n-p}\cap V_n$, or
\item $|C|$ is a multiple of $|V_k|$, where $k\geq p$ and $V_k=S_k\cap V_n$.
\end{enumerate}
\end{lemma}

\begin{proof} This follows from \cite[Lemma~2.7]{15Gor}.
\end{proof}

For the centralizer $C$ from Lemma~\ref{l:cenp}, we say that $C$ is {\em of the first type}, if $|C|$ as in item~(i) of the lemma, and $C$ is {\em of the second type} otherwise. As readily seen, the order of every centralizer of the first type is strictly less than the order of the centralizer of the second type.

\begin{lemma}\label{l:Dusart}
If $n\geq 11$, then the interval $(4n/5,n]$ contains a prime.
\end{lemma}

\begin{proof} For $11\leq n\leq 31$ this can be easily verified. For $n\geq32$ this follows from \cite{52Nag}.
\end{proof}

\begin{lemma}\label{l:divide}
Let $K$ be a normal subgroup of a group $G$ and $\ov{G}=G/K$. Let $\ov{x}$ be the image of $x\in G$ in $\ov{G}$. Then the following hold:
\begin{enumerate}
\item $|x^K|$ and $|\ov{x}^{\ov{G}}|$ divide $|x^G|;$
\item if $|x|$ and $|K|$ are coprime, then $C_{\ov{G}}(\ov{x})=C_G(x)K/K$.
\end{enumerate}
\end{lemma}

\begin{proof} Item (i) is well known. Item (ii) is \cite[Theorem~1.6.2]{93Khu}.
\end{proof}

Recall that a group $H$ is a section of a group $G$ if $H$ is a homomorphic image of some subgroup of~$G$.

\begin{lemma}\label{l:sectionAm}{\em\cite[Proposition 2.7]{PSV25}}
Suppose that $A_m$ is a section of $\operatorname{GL}_d(q)$ for some $d\geq 1$. Then $d\geq m-2$ for $m\geq9$.
\end{lemma}

Recall that a finite group (an element of a group) is a $\pi$-group (a $\pi$-element) if its order is a $\pi$-number. A subgroup $H$ is called a Hall $\pi$-subgroup of $G$ if $|H|$ is a $\pi$-number and $|G:H|$ is a $\pi'$-number. Following~\cite{56Hal},  we say that a finite group $G$ is a $D_\pi${\em-group} (or $G$ has the property $D_\pi$) if $G$ contains a Hall $\pi$-subgroup, all Hall $\pi$-subgroups of $G$ are conjugate, and every $\pi$-subgroup of $G$ is contained in some Hall $\pi$-subgroup of $G$. In this language, the celebrated Hall theorem claims that the solvable groups are $D_\pi$-groups \cite{28Hal}. The following assertion provides a useful criterion for a group to has the $D_\pi$-property.

\begin{lemma}{{\rm\cite[Theorem~6.6]{11VdRev.t}}}\label{l:dp}
Let $\pi$ be a set of primes, let $G$ be a group and $N$ a normal subgroup of $G$. The group $G$ is a $D_\pi$-group if and only if $N$ and $G/N$ are $D_\pi$-groups.
\end{lemma}

The following obvious observation on $D_\pi$-groups is quite helpful.

\begin{lemma}\label{l:ordpi}
Suppose that $G$ is a $D_\pi$-group and $H$ is a Hall $\pi$-subgroup of $G$. Then each $\pi$-element of $G$ is conjugate to some element of~$H$.
\end{lemma}

%

\section{Proof of the main result}

Let $V_n\in\{A_n, S_n\}$, let $G$ be a finite group with $\cs(G)=\cs(V_n)$ and $|G|=|V_n|$.

Observe that the theorem can be easily verified for $n\leq4$. So we assume that $n\geq5$. Since $Z(G)=1$ in view of Lemma~\ref{l:ZofG}, we may suppose that $n,n+1,n+2$ are not primes for $V_n=A_n$; and $n,n+1$ are not primes for $V_n=S_n$, because Thompson's conjecture holds true otherwise, see \cite{05Al}. Moreover, Thompson conjecture has been verified when $n=10,16,22,26$ and $V_n=A_n$ in \cite{09Vas}, \cite{12Gor}, \cite{13Xu}, \cite{14Liu}, respectively. Thus, we may suppose that $n\geq27$ for $V_n=A_n$ and $n\geq9$ for $V_n=S_n$.

Further, $p$ is the greatest prime not exceeding $n$.

\begin{prop}\label{p:GorMain}
If $n\geq23$, then there exists a normal subgroup $K$ of $G$ such that $A_m\leq \ov{G}=G/K\leq S_m$, where $m\in[p,n]$.
\end{prop}

\begin{proof} Lemma~\ref{l:ZofG} yields $Z(G)=1$. Therefore, the claim follows from the main result of \cite{19GorA}.
\end{proof}

\begin{prop}\label{p:smallSn}
If $n\geq11$, then there exists a normal subgroup $K$ of $G$ such that $A_m\leq \ov{G}=G/K\leq S_m$, where $m\in[p,n]$.
\end{prop}

\begin{proof} By the above arguments, we need to consider the groups $S_n$ for $n\in\{15,16,21,22\}$ only.

Let $n\in\{15,16\}$. Suppose that $x$ is an element of $G$ of order $p=13$ and $C=C_G(x)$ is of the second type, see Lemma~\ref{l:cenp}(ii). Then according to this lemma $p!$ divides $|C|$. Since $\cs(G)=\cs(S_n)$, there are only two possible numbers $i_1,i_2$ in $\cs(G)\setminus\{1\}$ such that $p!$ can divide the ratio $|G|/i_k$, $k=1,2$. These numbers are $i_1=n(n-1)/2$, so $i_1$ is equal to $3\cdot5\cdot7$ or $2^3\cdot3\cdot5$ for $n=15$ and $16$, respectively; and $i_2=n(n-1)(n-2)/3=2^5\cdot5\cdot7$ possible only for $n=16$. It follows that for $r=5$ there is a Sylow $r$-subgroup $R$ of $G$ with $M=R\cap C$ of index~$r$. Since $M$ is maximal in $R$, it intersects $Z(R)$ nontrivially. Therefore, there is an element $y\in Z(R)\cap C$. Then the index $|G:C_G(y)|$ must be coprime to $65$. Since $\cs(S_{15})$ and $\cs(S_{16})$ do not have such numbers, we derive a contradiction.

Thus, $C$ is of the first type, see Lemma~\ref{l:cenp}(i). In particular, each prime $s\in(n/2,n]$ distinct from $p$ does not divide the order of~$C$. Furthermore, since $|G|=|V_n|$, the numbers $|G|/sp$ and $sp$ are coprime. Suppose that $p$ and one of such $s$ divide the orders of distinct composition factors of $G$. Then Lemma~\ref{l:dp} implies that $G$ is a $D_\pi$-group for $\pi=\{s,p\}$. Therefore, $G$ has a Hall $\pi$-subgroup~$H$ which contains the element $x$, see Lemma~\ref{l:ordpi}. The subgroup $H$ being of order $sp$ must be cyclic, which contradicts the fact that $s$ does not divide the order of~$C$. Thus, there is a normal subgroup $K$ of $G$ such that $\ov{G}=G/K$ has a nonabelian simple subgroup $L$ whose order is a multiple of all primes from the interval $(n/2,n]$. In particular, $G$ is nonsolvable. Since the order of $L$ must divide the order of $V_n$, using \cite{09Zav}, one can easily verified that there are no such groups except $A_m$, where $m\in[p,n]$.

The case $n\in\{21,22\}$ can be considered in the same manner. Here $p=19$, and one can take $r=7$. Then again there is an element $y$ from the center of a Sylow $r$-subgroup $R$ of $G$ such that $y\in C=C_G(x)$, where $x$ is an element of order~$p$. This proves that $G$ is nonsolvable. Moreover, there is a normal subgroup $K$ of $G$ such that $\ov{G}=G/K$ has a nonabelian simple subgroup $L$ which order is a multiple of all primes from the interval $(n/2,n]$. It was checked that there are no such groups except $A_m$, where $m\in[p,n]$.
\end{proof}

\begin{prop}\label{p:S10}
The theorem holds for $V_n=S_n$ and $n=9,10$.
\end{prop}

\begin{proof} Let $x$ be an element of order $p=7$. Suppose first that $C=C_G(x)$ is of the second type.

Let $n=10$.  Since $p!$ divides $|C|$, there are only two numbers $i_1=n(n-1)/2=3^2\cdot5$ and $i_2=n(n-1)(n-2)/3=2^4\cdot3\cdot5$ in $\cs(G)\setminus\{1\}$ such that $p!$ divides the ratio $|G|/i_k$, $k=1,2$. It follows that there is a Sylow $5$-subgroup $R$ of $G$ intersecting $C$ nontrivially. Let $y\in R\cap C$. Since $R$ is of order $25$, it is abelian. Hence $|G:C_G(y)|$ must be coprime to $35$. Since the set $\cs(G)=\cs(S_{10})$ does not have such numbers, we derive a contradiction.

Suppose $n=9$. Since $p!$ divides $|C|$, we have $|x^G|=2^2\cdot 3^2$. Suppose that $G$ has a composition factor of order $7$. Then Lemma~\ref{l:dp} implies that $G$ is a $D_\pi$-group for $\pi=\{3,7\}$. Hence $G$ contains a Hall $\pi$-subgroup $H$ of order $3^4\cdot7$ and we may assume that $x\in H$. Since the order of $3$ module $7$ is $6$, the order of $C_H(x)$ is at least $3^3\cdot7$, but the $3$-part of $C_G(x)$ is equal to $3^2$, a contradiction. Therefore, $7$ divides the order of a nonabelian composition factor $L$ of $G$. Let $\ov{x}$ be the image of $x$ in $L$. Using \cite{Atlas} and getting through all possible nonabelian simple groups $L$ such that $|L|$ divides $|S_9|$, we see that the $2$-part of $|\ov{x}^L|$ is at least $2^3$. But $|\ov{x}^L|$ must divide $|x^G|$ in view of Lemma~\ref{l:divide}(i), a contradiction.

Thus, $C$ is of the first type. In particular, $|C|=7d$, where $d$ divides $2!$ in the case $n=9$ and $d$ divides $3!$ in the case $n=10$. If $G$ has a composition factor of order $7$, then $G$ is a $D_\pi$-group for $\pi$ equal to $\{5,7\}$ and $\{3,7\}$. However, a Hall $\{5,7\}$-subgroup of order $35$ for $n=9$ or $175$ for $n=10$ must be cyclic, which is impossible in our situation. It follows that $G$ is nonsolvable. Moreover, the order of a nonabelian composition factor $L$ of $G$, which contains an element of order $7$, is a multiple of  $3^4\cdot5\cdot7$ in the case $n=9$ and $3^3\cdot5^2\cdot7$ in the case $n=10$. The arguments are similar to the ones from the previous paragraph. Since $|L|$ divides $|S_n|$, it can be easily verified (see, e.g., \cite{09Zav}) that $L\simeq A_9$ for $n=9$ and $L\simeq A_{10}$ for $n=10$.

To complete the proof, it remains to deal with the case, when $G$ has a composition factor isomorphic to $A_n$ but $G\not\simeq S_n$. In this case, $G$ includes a normal subgroup $K$ with $G/K\simeq A_n$. Since $|G|=|S_n|$, the order of  $K$ equals~$2$, so and $Z(G)=K\neq1$, a contradiction.
\end{proof}

Thus, $n\geq15$ and there exists a normal subgroup $K$ of $G$ such that $A_m\leq \ov{G}=G/K\leq S_m$, where $m\in[p,n]$. Suppose further to the contrary that $G$ is not isomorphic to $V_n$.

\begin{lemma}\label{l:mlesen}
We have $m<n$.
\end{lemma}

\begin{proof} If $m=n$, then $V_n=S_n$ and $|K|=2$. It follows that $Z(G)\neq1$, a contradiction.
\end{proof}

\begin{lemma}\label{l:CGK}
$C_G(K)\leq K$.
\end{lemma}

\begin{proof} Assume the opposite. Then $A_m \leq C_G(K)K/K\leq S_m$. Hence there is $x\in C_G(K)$ of order $p$ such that its image $\ov{x}\in\ov{G}$ is a cycle of length $p$ in $A_m$. Observe that $|x|$ and $|K|$ are coprime, so Lemma~\ref{l:divide}(ii) yields $C_{\ov{G}}(\ov{x})=C_G(x)K/K\simeq C_G(x)/C_K(x)$. Since $|G|=|V_n|$ and $C_K(x)=K$, it follows that
$$
c=|C_G(x)|=|C_{\ov{G}}(\ov{x})||K|=\frac{p\,(m-p)!\,n!}{d\,m!},\mbox{ where }d=|V_n|/|A_n|.
$$

Denote by $c_1$ the maximum of the orders of the centralizers of first type from item (i) of Lemma~\ref{l:cenp}. Since $m<n$, we have $c_1<c$. On the other hand, $p\geq4n/5$ in view of Lemma~\ref{l:Dusart}. Hence $c<p!$, so $c$ is less than the order of any centralizer of the second type from item (ii) of Lemma~\ref{l:cenp}. This is a contradiction, because $|G|/c\in\cs(G)=\cs(V_n)$.
\end{proof}

\begin{lemma}\label{l:RinK}
Suppose that $R$ is a Sylow $r$-subgroup of $K$ for some prime $r$. Then $$|R|<n\cdot r^{n/5}.$$
\end{lemma}

\begin{proof} Let $t=n/5$ and $s=[\log_rn]$. In view of Lemma~\ref{l:Dusart}, we have $n-m\leq t$.  It follows from Legendre's formula that
$$
|R|=\left(\frac{n!}{m!}\right)_r=r^{\sum_{i=1}^s([n/r^i]-[m/r^i])}\leq r^{s+\sum_{i=1}^s[t/r^i]}\leq n\cdot r^{\sum_{i=1}^s[t/r^i]}<n\cdot r^{t/(r-1)}\leq n\cdot r^{t}.
$$
\end{proof}

We are ready to finish the proof of the theorem. Lemma~\ref{l:CGK} implies that there is a prime divisor $r$ of the order of~$|K|$ such that $N_G(R)/C_G(R)$ includes a section isomorphic to $A_m$, where $R$ is a Sylow $r$-subgroup of $K$. Now Lemma~\ref{l:sectionAm} yields $|R|\geq r^{m-2}$. Since $n\geq15$, it follows from Lemma~\ref{l:Dusart} that $m\geq4n/5$. Applying Lemma~\ref{l:RinK}, we obtain
$n\cdot r^{n/5}>r^{4n/5-2}$, so $n>r^{(3n-10)/5}$. This gives us a contradiction for $n\geq15$. The theorem is proved.

\medskip

\textbf{Acknowledgments.} The idea of this article appeared when the authors visited the Sino-Russian Mathematics Center of the Peking University. The authors express their deep gratitude to the director of the SRMC Professor Zhang Jiping for the invitation, warm hospitality and fruitful discussions. The research was carried out within the framework of the Sobolev Institute of Mathematics project FWNF-2026-0017.

\end{document}